\documentclass{amsart}

\usepackage[T1]{fontenc}
\usepackage[utf8]{inputenc}
\usepackage{amsmath,amsthm,amsfonts,amscd,amssymb,eucal,latexsym,mathrsfs}
\usepackage{stmaryrd}
\usepackage{enumerate}
\usepackage{hyperref}
\usepackage[all]{xy}
\usepackage{etoolbox}
\usepackage{amssymb}
\usepackage{amsfonts}
\usepackage{graphicx}
\usepackage{mathtools}

\usepackage{tikz,tikz-cd}
\usetikzlibrary{shapes.geometric}
\usetikzlibrary{arrows}
\usetikzlibrary{calc}

\usetikzlibrary{positioning}
\usepackage{float}
\usepackage{MnSymbol}
\usetikzlibrary{matrix}
\usepackage{tkz-euclide}

\theoremstyle{plain}
\newtheorem{theorem}{Theorem}[section]
\newtheorem*{theorem*}{Theorem}

\newtheorem{corollary}[theorem]{Corollary}
\newtheorem*{corollary*}{Corollary}

\newtheorem{lemma}[theorem]{Lemma}

\theoremstyle{definition}

\newtheorem*{definition*}{Definition}

\usetikzlibrary{calc,graphs}
\usepackage{xcolor}
\usetikzlibrary{arrows,decorations.pathmorphing}

\newcommand{\e}{\varepsilon}

\newcommand{\IF}{\mathbb{F}}

\newcommand{\ZI}{\mathbb{Z}}

\DeclareMathOperator{\surj}{\twoheadrightarrow}

\newcommand*\mat[1]{\protect\begin{pmatrix}#1\protect\end{pmatrix}}

\newcommand{\SL}{\mathrm{SL}}

\newcommand{\ts}{\textsection}

\newcommand{\ip}[1]{\langle#1\rangle} 

\title{On the kissing number of hyperbolic surfaces with a small systole}
\author{Sylvain Barré, Mika\"el Pichot}

\newcommand{\Kiss}{\mathrm{Kiss}}
\newcommand{\PGL}{\mathrm{PGL}}

\begin{document}

\maketitle

\begin{abstract}
Parlier proved that there exists a constant $U>0$ such that for any hyperbolic surface $S$ of genus $g$ with systole $\ell$ the following holds:
\[
\Kiss(S)\leq U \frac {e^\ell}\ell g
\]
where $\Kiss(S)$ is the maximum number of homotopically distinct systoles in $S$ (Kissing number). 

We  show that there exists a constant $V>0$ and a sequence $S_n$ of hyperbolic homogeneous surfaces of genus $g_n\to \infty$ with a bounded systole such that 
\[
\Kiss(S_n)\geq Vg_n
\] 
for all $n\geq 1$.  This shows that Parlier's theorem is sharp for a fixed systole length.

Our proof uses Bruhat--Tits buildings of type $\tilde A_2$ and Hamiltonian surfaces.
\end{abstract}

\section{Introduction}

In his paper \cite{P}, Hugo Parlier proves that the number of homotopically distinct systoles (kissing number) in a closed hyperbolic surface $S$ of genus $g$ and systoles $\ell$ is bounded above by
\[
\Kiss(S)\leq U \frac {e^\ell}\ell g
\]
where $U>0$ is a constant. This implies that the number of systoles in any hyperbolic surface of genus $g$ is bounded above by $\frac {g^2}{\log g}$.

The search for lower bounds has a long history (see \cite{P}).  Paul Schmutz Schaller \cite{sc,sc2} has shown that there exists a familly of arithmetic surfaces with kissing number bounded below by
\[
\Kiss(S)\geq g^{4/3-\e}\ \ \forall \e>0.
\] 
This is the best known construction to date and therefore there is a gap between Schmutz Schaller's lower bound $g^{4/3-\e}$ and Parlier's upper bound $\frac {g^2}{\log g}$. In \cite{sc2}, Schmutz Schaller conjectures that the bound $g^{4/3}$ is optimal. 

Our main theorem shows that Parlier upper bound is sharp in the case of a surface with a bounded systole.  

\begin{theorem}\label{T - th1} 
There exists a constant $V>0$ and a sequence $S_n$ of hyperbolic homogeneous surfaces of genus $g_n\to \infty$ with a bounded systole such that 
\[
\Kiss(S_n) \geq V g_n
\] 
for all $n\geq 1$.\end{theorem}

The systole in these surfaces is in fact constant in $n$ equal to $\ell\approx 8.69$.

The proof of the theorem relies on the construction of a certain type of regular surfaces that are associated with a lattice in $\PGL_3(\IF_2((T)))$.  These surfaces are ``Hamiltonian surfaces'', in the sense of \cite{ham}. They will visit every vertex and every edge of a large quotient of the Bruhat--Tits building associated with    $\PGL_3(\IF_2((T)))$, exactly once. 
 
The notion of Hamiltonian surfaces derives from a new construction of 2-complexes using  an operation of ``systolic filling'', in a finite covering space associated with a given surface of infinite genus. This approach was first studied in  \cite{torus} for systolic fillings associated with a pinched torus, and later in \cite{ham}, for systolic fillings of hyperbolic surfaces of infinite genus, naturally associated with the group $\mathrm{Aut}(F_2)$.

\subsection{On AI usage} No use of AI was made in the conception or the writing of this paper.

\section{Construction of the Hamiltonian surface}

By Hamiltonian surface in a 2-complex $X$, we mean a surface which visits every vertex and every edge precisely once. A Hamiltonian surface is said to be periodic if it endowed with a free and cocompact action of a countable group. 

\begin{theorem}\label{T - ham A2}
The Bruhat--Tits building $X$ of type $\tilde A_2$ over the field $\IF_2((T))$ contains three distinct periodic Hamiltonian surfaces, which are cyclically permuted by a group $G$ of automorphisms of $X$ that acts simply transitively by type-rotating automorphisms. 
\end{theorem}

The group $G$ in this theorem was introduced in \cite{CMSZ1}. It will be described explicitly in the proof.

\begin{proof}
We first review the construction of an explicit embedding into $\mathrm{PGL}_3(\IF_2((T)))$ of an amalgamated group (of Tits type) following the work of \cite{Tits, KMW, Ro, CMSZ1, CMSZ2}. Namely, we consider the Frobenius group of order 21  defined by:
\[
F=\ip{p,s\mid p^7=s^3=1, sps^{-1}=p^2}\simeq \ip{p,s\mid p^7=s^3=1, sps^{-1}=p^4}
\]
 in which $\ZI/3\ZI$ acts on $\ZI/7\ZI$ by conjugacy.  It is well--known \cite{Ro} that two distinct copies of $\ZI/3\ZI$ in $F$ generate a chamber system for the building of $\SL_3(2)$.

Furthermore, Tits has shown in \cite{Tits} (see also \cite[p.\ 48]{Ro}) that there are three ways to amalgamate three copies of $F$ together along two distinct copies of $\ZI/3\ZI$ in each group, in order to define a group, which does not collapse, and acts by automorphisms of the Bruhat--Tits building of $\mathrm{PGL}_3(\IF_2((T)))$. 

In \cite{KMW}, Kohler, Meixner, and Wester have shown that one of these three amalgams  embeds in $\mathrm{PGL}_3(\IF_2((T)))$. Furthermore, they use this embedding to construct exponentially many pairwise non-isomorphic finite transitive chamber systems of type $\tilde A_2$. It will be convenient to use a closely related embedding of the following group $G$, which was introduced by Cartwright, Mantero, Steger,  and  Zappa in \cite{CMSZ1,CMSZ2}. Thus, we consider the group $G$ on  seven generators $a_i$ and seven relations:
\[
a_ia_{i+1}a_{i+3}
\]
where $i\in \ZI/7\ZI$. This presentation is an $\tilde A_2$-triangle presentation in the sense of \cite[p.\ 149]{CMSZ1}. The automorphism group of this triangle presentation is $F$ and the extension $\tilde G$ of $G$ by $F$ is generated by the $a_i$, $p$ and $s$ with commutation relations:
\[
pa_ip^{-1}=a_{p(i)}\ \  \ sa_is^{-1}=a_{s(i)}\ \ \
\]
where $p(i)=i+1$ and $s(i)=2i$.

The following explicit embedding theorem is established in \cite[\ts 4, Example]{CMSZ1} and described in \cite[p.\ 177]{CMSZ2}. 

\begin{theorem}
The map 
\[
a_0\mapsto \mat{T&0&0\\ 0&T&1+T\\0&1+T&1}, \ \ \ \ p\mapsto \mat{1&0&1\\ 1&0&0\\0&1&0},\ \ \ \ s\mapsto \mat{1&0&0\\ 0&1&1\\0&1&0}
\]
desfines to  an embedding of $\tilde G$ into $\mathrm{PGL}_3(\IF_2((T)))$.
\end{theorem}

The relation with the embedding of Kohler, Meixner, and Wester is explained in the last part of \ts 4 in \cite{CMSZ1}.

It is shown in \cite{CMSZ1} that the group $G$ acts simply transitively on the vertices of the building $X$ of $\mathrm{PGL}_3(\IF_2((T)))$. Furthermore, the action is by type-rotating automorphisms. In fact, the type-rotating  assumption essentially yields a presentation of $G$ of the form given above, namely the $\tilde A_2$-triangle presentation of $G$. For all this, we refer \cite[Theorem 3.1]{CMSZ1}. 

The type-rotating assumption implies that the group $G$ admits a subgroup $G_3$ of index 3, which act by type-preserving automorphisms. Furthermore, the presentation of $G_3$ is easily computed, following \cite{CMSZ1}.  In the particular case above, we find a group on 21 generators $a_{i,j}$ and 21 relations, as follows:
\[
a_{i,j}a_{i+1,j+1}a_{i+3,j+2}
\]
where $i\in \ZI/7\ZI$ and $j\in \ZI/3\ZI$. We let $X_3$ denote the corresponding finite chamber system.

Our main observation is: 

\begin{lemma}\label{L - surface Sp}
 The set $\Sigma_p$ of 14 chambers $a_{i,j}a_{i+1,j+1}a_{i+3,j+2}$ for all $i\in \ZI/7\ZI$ and all  $j$ omitting the value $p$ in $\ZI/3\ZI$ defines a Hamiltonian surface in $X_3$.
\end{lemma}

\begin{proof} 
Since one value of $j$ is omitted, it is clear that the generator $a_{ij}$ appears only twice in the given set of chambers, for all $i\in \ZI/7\ZI$ and $j\in\ZI/3\ZI$. To prove that $X_p$ is a hamiltonian surface, it suffices to prove that its links form a Hamilonian cycle in the links of $X_3$. 

The generic chamber $C_{ij}$ in $\Sigma_p$ is given by:

\begin{figure}[H]
\begin{tikzpicture}[
    scale=1, 
    vertex/.style={circle, fill=black, inner sep=1.5pt},
    midarrow/.style={midway, sloped, allow upside down, inner sep=0pt, node contents={\tikz\draw[-Latex] (0,0) -- (0.1,0);}}
]
    \coordinate (A) at (0,0);
    \coordinate (B) at (2,0);
    \coordinate (C) at (1,1.732);

    \draw (A) -- (B) node[midarrow] node[midway, below=4pt] {$a_{i,j}$};
    \draw (B) -- (C) node[midarrow] node[midway, right=5pt] {$a_{i+1,j+1}$};
    \draw (C) -- (A) node[midarrow] node[midway, left=5pt]  {$a_{i+3,j+2}$};

    \node[vertex] at (A) {};
    \node[vertex] at (B) {};
    \node[vertex] at (C) {};
\end{tikzpicture}
\caption{Chamber $C_{ij}$}
\end{figure}
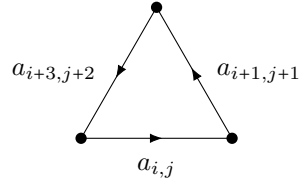

By definition, the link of $\Sigma_p$ is the graph with edges:
\[
(a_{ij}^-,a_{i+1,j+1}^+), (a_{i+1,j+1}^-,a_{i+3,j+2}^+),(a_{i+3,j+2}^-,a_{ij}^+)
\] 
where $i\in \ZI/7\ZI$ and $j\in \ZI/3\ZI\setminus p$. It is straightforward to verify that this graph is a union of three 14-gons. 
\end{proof}

To conclude the proof of the theorem, it suffices to prove that $G$ permutes cyclically the three lifts of the surfaces $\Sigma_p\subset X_3$ to the universal cover $X$. This follows because the action is type-rotating. It is clear that each lift is a periodic Hamiltonian surface in $X$ on which $G_3$ act by automorphisms.
\end{proof}

The following result was stated without proof in \cite{ham}, we can now provide a proof. 

\begin{corollary}
The Bruhat--Tits building of type $\tilde A_2$ over $\IF_2((T))$ contains uncountably many Hamiltonian surfaces.
\end{corollary}

\begin{proof}
If $K$ denotes a vertex stabilizer in $\mathrm{PGL}_3(\IF_2((T)))$ and $\Sigma$ is one of the three Hamiltonian surfaces in $X$, then the orbit $K\Sigma$ is uncountable, by the orbit--stabilizer theorem. 
\end{proof}

\section{Proof of Theorem \ref{T - th1}}

We establish the following theorem which is our main result. 

\begin{theorem}\label{T - th3} 
There exists a sequence $S_n$ of simplicially hyperbolic homogeneous surfaces of genus $g_n\to \infty$ with bounded systole $\ell=3$ such that 
\[
\Kiss(S_n) = \frac 7 2(g_n-1)
\] 
for all $n\geq 1$.\end{theorem}

We note that Theorem \ref{T - th1} follows immediately from this result by replacing every equilateral triangle in $S_n$ with the corresponding hyperbolic triangle of constant curvature and vertex angle $\frac{\pi}{7}$. Each systole of length 3 gives rise to a systole of length
\[
\ell= 3\times 2\mathrm{arccosh}(\frac{1}{2\sin(\pi/14)})\simeq  8.69
\]
in the corresponding smooth surface.

\begin{proof}
Let $X$ denote the Bruhat--Tits building of $\PGL_3(\IF_2((T)))$. Consider one of the three periodic hamiltonian surface $S$ in $X$ defined in Theorem  \ref{T - ham A2}. This surface is $G_3$-periodic, where $G_3$ is of index 3 in $G$.  We will need to use here the group $G$ defined in \cite{CMSZ1} and reviewed in the previous paragraph. 

Since $G_3$ is linear, it is residually finite. Therefore, there exists a finite index subgroup $K_0$ of $G_3$ for which the injectivity radius of the cover space $X\surj X/K_0$ is at least 2. This implies that the systole in $S/K_0$ is at least 3.

Let $(K_n)_{n\geq 0}$ be a decreasing sequence of normal finite index normal subgroups of $G_3$, and let $S_n:=S/K_n$.  The 2-complex $X_n:=X/K_n$ is obtained from $S_n$ by systolic filling in the sense of \cite{ham}. The systoles of $S_n$ have length 3, and $X_n$ is obtained canonically from $S_n$ by adjoining a triangle for every systole. This is in fact easy to see directly from the construction of $S$ in Theorem \ref{T - ham A2}.     Namely, in Lemma \ref{L - surface Sp}, the surface  $\Sigma_p$ is obtained by omitting the chambers corresponding to a single index $p\in \ZI/3\ZI$. The systolic filling of $S_n$ is obtained by adjoining all chambers which are lifts of the chambers corresponding to this index. In this construction, every systole is a boundary of such a chamber. It has length 3, and meets every edge precisely once. The residual finiteness of $G$ and the presence of $X$ is crucial to find the correct covering space of $S/G$.

It is clear that the surfaces $S_n$ are simplicially hyperbolic with (by construction) 14 equilateral triangles meeting at every vertex. Furthermore, they are homogeneous surfaces acted upon by the group $H_n:=G/K_n$

It remains to compute the number of systoles and the genus of $S_n$. Since every systole meets an edge exactly once, the number of systoles is $e_n/3$ where $e_n$ is the number of edges of $S_n$. Since $S_n$ is Hamiltonian, the number of edges in $S_n$ equals the number of edges in $X_n$.

Since $G$ acts freely transitively by type rotating automorphisms on $X$, and $G_3$ acts freely by type preserving automorphisms of $X$, the action of $G_3$ on $X_3$ is free with $3$ orbits of vertices, $21$ orbits of edges, and $21$ orbits of faces. Furthermore, the action $H_n$ on $X_n$ is free with quotient complex $X/G_3$.  In particular, the surface $S_n$ contains precisely  $e_n=21|H_n|$ edges and 
therefore $7|H_n|$ systoles. Furthermore, the Euler characteristic of $S_n$ is
 \[
2-2g_n= \chi(S_n)=(v_0-e_0+f_0)|H_n|=-4|H_n|.
 \]
 This shows that
 \[
 \Kiss(S_n)=7|H_n|=\frac 7 2(g_n-1)
 \]
for all $n\geq 1$.
\end{proof}

\end{document}